\documentclass[11pt]{amsart}

\usepackage[margin=2cm]{geometry}

\usepackage{latexsym, amssymb, amsfonts, amsmath, amsthm, graphics, graphicx, subfigure, bbm, stmaryrd,xcolor}
\usepackage{epstopdf}
\usepackage[pdftex,colorlinks,linkcolor=blue,menucolor=red,backref=false,bookmarks=true,citecolor=olive]{hyperref}

\usepackage{palatino}

\newtheorem{theorem}{Theorem}

\newtheorem{corollary}[theorem]{Corollary}

\newtheorem{proposition}[theorem]{Proposition}

\newcommand{\field}[1]{\mathbb{#1}}
\newcommand{\R}{\field{R}}

\newcommand{\var}{\mathrm{Var}}

\numberwithin{equation}{section}

\begin{document}

\title{A refinement of Brascamp-Lieb-Poincar\'e inequality in one dimension}

\begin{abstract}  
In this short note we give a refinement of Brascamp-Lieb \cite{1} in the style of Houdr\'e-Kagan \cite{3} extension for Poincar\'e inequality in one dimension.   This is inspired by works of Helffer \cite{2} and Ledoux \cite{4}.    
\end{abstract}

\author{Ionel Popescu}
\address{School of Mathematics, Georgia Institute of Technology, 686 Cherry Street, Atlanta, GA 30332, USA}  \address{``Simion Stoilow'' Institute of Mathematics   of Romanian Academy, 21 Calea Grivi\c tei, Bucharest, ROMANIA}
\email{ipopescu@math.gatech.edu,  ionel.popescu@imar.ro}

\thanks{The author was partially supported by a grant of the Romanian National Authority for Scientific Research, CNCS - UEFISCDI, project number PN-II-RU-TE-2011-3-0259  and by European Union Marie Curie Action Grant PIRG.GA.2009.249200.}

\maketitle

\section{The Brascamp-Lieb inequality}

We take a convex potential $V:\R\to\R$ which is $C^{k}$ with $k\ge2$ and the measure $\mu(dx)=e^{-V(x)}dx$ which we assume it is a probability measure on $\R$.   
\begin{theorem}[Brascamp-Lieb \cite{1}]\label{t:1}
If $V''>0$ then for any $C^{2}$ compactly supported function $f$ on the real line 
\begin{equation}\label{eq:1}
\var_{\mu}(\phi)\le \int \frac{(f')^{2}}{V''}d\mu.
\end{equation}
\end{theorem}

One of the proofs is due to Helffer \cite{2} and we sketch it here as it is the starting point of our approach.   

Consider the operator $L$ acting on $C^{2}$ functions is given by 
\[
L=-D^{2}+V'D
\]
with $D\phi=\phi'$.  We denote $\langle \cdot,\cdot \rangle$ the $L^{2}(\mu)$ inner product and observe that
\[
\langle L\phi,\phi \rangle =\|  \phi'\|^{2}.
\]

In particular $L$ can be extended to an unbounded non-negative operator on $L^{2}(\mu)$.  
From this, we get 
\begin{equation}\label{e1:0}
\| L\phi\|^{2}=\langle D L\phi, D\phi\rangle
\end{equation}
and then if we take $f$ a $C^{2}$ compactly supported function such that $\int fd\mu=0$ and replace $\phi=L^{-1}f$, then we get 
\[
\var_{\mu}(f)=\langle f',DL^{-1}f \rangle.
\] 
Now a simple calculation reveals that 
\[
DL =(L+V'')D
\]
and then $(L+V'')^{-1}D=DL^{-1}$ where the inverses are defined appropriately.  Therefore we get 
\begin{equation}\label{eq:2}
\var_{\mu}(f)=\langle (L+V'')^{-1}f',f' \rangle.
\end{equation}
Since $L$ is a non-negative operator $(L+V'')^{-1}\le (V'')^{-1}$ and this implies \eqref{eq:1}.

\section{Refinements in the case of $\R$}

We start with \eqref{eq:2} and iterate it.   This is inspired from \cite{4} but without any use of the semigroup theory.  

We let $D$ be the derivation operator and we denote $D^{*}=-D+V'$, the adjoint of $D$ with respect to the inner product in $L^{2}(\mu)$.   In the sequel, for a given function $F$, we are going to denote by $F$ also the multiplication operator by $F$.   The main commutation relations are the content of the following.

\begin{proposition} Let $\mathcal{A}$ denote the operator defined on smooth positive functions $E$ given by 
\begin{equation}\label{e2:1}
\mathcal{A}(E)(x)=\frac{1}{4} \left(2 E''(x)+2 V'(x)
   E'(x)-\frac{E'(x)^2}{E(x)}+4
   E(x) V''(x)\right) 
\end{equation}
\begin{enumerate}
\item If $E$ is a positive function, then 
\begin{equation}\label{e2:2}
DED^{*}=\mathcal{A}(E)+E^{1/2}D^{*}DE^{1/2}
\end{equation}
\item For a positive function $E$, 
\begin{equation}\label{e2:3}
(E+D^{*}D)^{-1}=E^{-1}-E^{-1}D^{*}(I+DE^{-1}D^{*})^{-1}DE^{-1}
\end{equation}

\item If $E$ is a positive function such that $1+\mathcal{A}(E^{-1})$ is positive and $F=E(1+\mathcal{A}(E^{-1}))$, then 
\begin{equation}\label{e2:6}
(I+DE^{-1}D^{*})^{-1}=F^{-1}E-E^{1/2}F^{-1}D^{*}(I+DF^{-1}D^{*})^{-1}DF^{-1}E^{1/2}.
\end{equation}
\end{enumerate}
\end{proposition}

\begin{proof}  
\begin{enumerate}
\item    We want to find two functions $F$ and $G$ such that 
\[
DED^{*}=F+GD^{*}DG
\]
For this, take a function $\phi$ and write 
\[
(DE(-D+V'))\phi
=(EV')'\phi+(-E'+EV')\phi'-E\phi''
\]
while
\[
\begin{split}
F\phi+G(-D+V')DG\phi
=(F-GG''+GG'V')\phi+(G^{2}V'-2GG')\phi'-G^{2}\phi''
\end{split}
\]
therefore it suffices to choose $G$ such that 
\[
G^{2}=E \text{ and } F=GG''-GG'V'+(EV')'
\]
which means $G=E^{1/2}$ and $F=\mathcal{A}(E)$.  

\item We have 
\[
\begin{split}
(E+D^{*}D)^{-1}&=E^{-1}-E^{-1/2}(I-(I+E^{-1/2}D^{*}DE^{-1/2})^{-1})E^{-1/2} \\
&=E^{-1}-E^{-1}D^{*}(I+DE^{-1}D^{*})^{-1}DE^{-1}
\end{split}
\]
where  we used the fact that for any operator $T$, 
\[
I-(I+T^{*}T)^{-1}=T^{*}(I-T T^{*})^{-1}T.
\]

\item From \eqref{e2:2}, we know that $I+DE^{-1}D^{*}=I+\mathcal{A}(E^{-1})+E^{-1/2}D^{*}DE^{-1/2}=FE^{-1}+E^{-1/2}D^{*}DE^{-1/2}$ and from \eqref{e2:3}, 
\[
(FE^{-1}+E^{-1/2}D^{*}DE^{-1/2})^{-1}=E^{1/2}(F+D^{*}D)^{-1}E^{1/2}=F^{-1}E-E^{1/2}F^{-1}D^{*}(I+DF^{-1}D^{*})^{-1}DF^{-1}E^{1/2}.
\]
\end{enumerate}
\end{proof}

Now, let us get back to the fact that $L=D^{*}D$ and that \eqref{eq:2} gives 
\[
\var_{\mu}(f)=\langle (V''+D^{*}D)^{-1}f',f' \rangle.
\]

From \eqref{e2:3} with  $E_{1}=V''$ we obtain first that  
\begin{equation}\label{e2:4}
\var_{\mu}(f)=\langle (V'')^{-1}f', f'\rangle-\langle (I+DE_{1}^{-1}D^{*})^{-1} D[E_{1}^{-1}f'],D[E_{1}^{-1}f'] \rangle.
\end{equation}

It is interesting to point out that this provides the case of equality in the Brascamp-Lieb if $D[(V'')^{-1}f']=0$ which solves for $f=C_{1}V'+C_{2}$. 


Now we want to continue the inequality in \eqref{e2:4} by taking $E_{1}=E$ and using \eqref{e2:4} for the case of $E_{2}=E_{1}(I+\mathcal{A}(E_{1}^{-1}))>0$, thus we continue with
\[
(I+DE_{1}^{-1}D^{*})^{-1} = E_{2}^{-1}E_{1}-E_{1}^{1/2}E_{2}^{-1}D^{*}(I+DE_{2}^{-1}D^{*})^{-1}DE_{2}^{-1}E_{1}^{1/2} .
\]
Hence we can write by setting $f_{1}=E_{1}^{-1}f'$ and $f_{2}=E_{1}^{1/2} D[f_{1}]$
\[
\begin{split}
\var_{\mu}(f)&=\| E_{1}^{-1/2}f'\|^{2} - \| E_{2}^{-1/2}f_{2}\|^{2} + \langle (I+DE_{2}^{-1}D^{*})^{-1}D[E_{2}^{-1}f_{2}],D[E_{2}^{-1}f_{2}] \rangle. 
\end{split}
\]
Using a similar argument, let $E_{3}=E_{2}(1+\mathcal{A}(E_{2}^{-1}))$ provided that $E_{3}$ is positive.  Then we can continue with 
\[
(I+DE_{2}^{-1}D^{*})^{-1}=I+\mathcal{A}(E_{2}^{-1})-E_{2}^{1/2}E_{3}^{-1}D^{*}(I+DE_{3}^{-1}D^{*})^{-1}DE_{3}^{-1}E_{2}^{1/2} 
\]
and letting $f_{3}=E_{2}^{1/2}D[f_{2}]$, we obtain 
\[
\var_{\mu}(f)=\| E_{1}^{-1/2}f'\|^{2} - \| E_{2}^{-1/2}f_{2}\|^{2} + \| E_{3}^{-1/2}f_{3}\|^{2} - \langle (I+DE_{3}^{-1}D^{*})^{-1}D[E_{3}^{-1}f_{3}],D[E_{3}^{-1}f_{3}] \rangle. 
\]
By induction we can define 
\begin{align}
E_{1}=V'' &\text{ and }    f_{1}=E_{1}^{-1}f' \\
E_{n}= E_{n-1}(1+\mathcal{A}(E_{n-1}^{-1})) &\text{ and }  f_{n}=E_{n-1}^{1/2}D[f_{n-1}].
\end{align}
Notice that here $E_{n}$ is defined only if $E_{n-1}$ is defined and positive and we will assume the sequence is defined as long as this condition is satisfied.     We get the following result. 
\begin{theorem} If $E_{1},E_{2},\dots, E_{n}$ are positive functions, then for any compactly supported function $f$,
\begin{equation}\label{e2:5}
\begin{split}
\var_{\mu}(f)&=\| E_{1}^{-1/2}f'\|^{2} - \| E_{2}^{-1/2}f_{2}\|^{2} +\dots + (-1)^{n-1}\| E_{n}^{-1/2}f_{n}\|^{2} \\ 
&\quad+(-1)^{n} \langle (I+DE_{n}^{-1}D^{*})^{-1}D[E_{n}^{-1}f_{n}],D[E_{n}^{-1}f_{n}] \rangle.
\end{split}
\end{equation}
In particular, for $n$ even, 
\[
\var_{\mu}(f)\ge\| E_{1}^{-1/2}f'\|^{2} - \| E_{2}^{-1/2}f_{2}\|^{2} +\dots + (-1)^{n-1}\| E_{n}^{-1/2}f_{n}\|^{2} 
\]
and for $n$ odd, 
\[
\var_{\mu}(f)\le\| E_{1}^{-1/2}f'\|^{2} - \| E_{2}^{-1/2}f_{2}\|^{2} +\dots + (-1)^{n-1}\| E_{n}^{-1/2}f_{n}\|^{2}.
\]
\end{theorem}

For $V(x)=x^{2}/2-\log(\sqrt{2\pi})$ this leads to the following version of Houdr\'e-Kagan \cite{3} due to Ledoux \cite{4}.
\begin{corollary} For $V(x)=x^{2}/2-\log(\sqrt{2\pi})$ and  $f$ which is $C^{n}$ with compact support, the following holds true
\[
\var_{\mu}(f)=\|f'\|^{2} - \frac{1}{2!}\| f''\|^{2} +\dots + \frac{(-1)^{n-1}}{(n-1)!}\| f^{(n-1)}\|^{2} +\frac{(-1)^{n}}{(n-1)!} \langle (n+L)^{-1}f^{(n)}, f^{(n)} \rangle.
\]
\end{corollary}

Another particular case is the following which is a reverse type Brascamp-Lieb.    

\begin{corollary}
\[
\var_{\mu}(f)\ge\langle (V'')^{-1}f',f'\rangle- \langle (1+\mathcal{A}((V'')^{-1}))^{-1}D[(V'')^{-1}f'],D[(V'')^{-1}f'] \rangle
\]
provided $1+\mathcal{A}((V'')^{-1})>0$ which is equivalent to 
which amounts to 
\begin{equation}\label{e:3:4}
3 V^{(3)}(x)^2+8 V''(x)^3-2 V^{(4)}(x)V''(x)-2V^{(3)}(x)
   V''(x)V'(x) >0.
\end{equation}

\end{corollary}

For instance in the case $a,b>0$ and 
\[
V(x)=ax^{2}/2+bx^{4}/4 +C
\]
(where $C$ is the normalizing constant which makes $\mu$ a  probability) the condition \eqref{e:3:4} reads as
\[\tag{*}
2 a^3-3 a b+\left(15 a^2 b+18 b^2\right) x^2+42 a b^2 x^4+45 b^3 x^6>0
\]
for any $x$.  In particular, for $x=0$, this gives $3b<2a^{2}$ which turns out to be enough to guarantee (*) for any other $x$.   For the next corrections the condition that $1+\mathcal{A}(E_{2}^{-1})>0$ becomes equivalent to 
\[\begin{split}
&4 a^9-18 a^7 b+27 a^3 b^3+\left(90 a^8 b-225 a^6 b^2+504 a^4 b^3+540 a^2 b^4\right) x^2+\left(916 a^7 b^2-756 a^5 b^3+4203 a^3 b^4-162 a b^5\right) x^4 \\
&+\left(5563 a^6 b^3+2172 a^4 b^4+11124 a^2 b^5+1944 b^6\right) x^6+\left(22326 a^5 b^4+23868 a^3 b^5+7209 a b^6\right) x^8 \\ 
&+\left(61689 a^4 b^5+74817 a^2 b^6-5832 b^7\right) x^{10}+\left(117864 a^3 b^6+109026 a b^7\right) x^{12}+\left(150741 a^2 b^7+63180 b^8\right) x^{14}\\ 
&+117450 a b^8 x^{16}+42525 b^9 x^{18}>0
\end{split}
\]
for all $x$.   This turns out be equivalent to  $b<\frac{1}{3} \left(-1+\sqrt{3}\right) a^2$.    In general, for higher corrections the condition $E_{n}>0$ appears to be equivalent to a condition of the form $b<a^{2}t_{n}$ for some $t_{n}>0$ which is decreasing in $n$ to $0$.  We do not have a solid proof of this, but some numerical simulations suggests this conclusion.   

Another example is the potential $V(x)=x^{2}/2-a\log(x^{2})+C$ with $a>0$, for which condition \eqref{e:3:4} becomes equivalent to 
\[
4 a^3 - 3 a x^2 + 12 a^2 x^2 + 7 a x^4 + x^6>0
\]
for all $x$.  This turns out to be equivalent to $a>a_{0}$, where $a_{0}$ is the solution in $(0,1)$ of the equation $108 - 855 a + 144 a^2 + 272 a^3=0$ and numerically is $a_{0}\approx0.129852$.   For the second order correction   a numerical simulation indicates that we need to take $a>a_{1}$ with $a_{1}\approx 0.314584$.  Some numerical approximations suggest that $E_{n}>0$ is equivalent to $a>a_{n}$ with $a_{n}$ being an increasing sequence to infinity.

\section*{Acknowledgements} The author wants to thank Michel Ledoux for an interesting conversation on this subject and to the reviewer of this paper for comments which led to an improvement of this note.

\bibliographystyle{plain}

\end{document}